\documentclass[12pt,reqno]{amsart}
\usepackage{amsfonts, amssymb, mathrsfs, amsmath, amsthm}
\usepackage{fullpage} 
\usepackage[colorlinks=true]{hyperref}
\usepackage{appendix}
\usepackage{bbm}

\usepackage{comment}
    \usepackage{url}
    \usepackage{arydshln}

\makeatletter
\let\ams@starttoc\@starttoc
\makeatother
\usepackage[parfill]{parskip}
\makeatletter
\let\@starttoc\ams@starttoc
\patchcmd{\@starttoc}{\makeatletter}{\makeatletter\parskip\z@}{}{}
\makeatother
\usepackage[colorlinks=true]{hyperref}
\usepackage{cleveref}
\usepackage{thm-restate}
\theoremstyle{definition}

\newtheorem{definition}{Definition}[section]
\newtheorem{theorem}[definition]{Theorem}
\newtheorem{lemma}[definition]{Lemma}
\newtheorem{corollary}[definition]{Corollary}
\newtheorem{proposition}[definition]{Proposition}
\newtheorem{remark}[definition]{Remark}
\newtheorem{example}[definition]{Example}
\newtheorem{conjecture}[definition]{Conjecture}
\newtheorem{problem}[definition]{Problem}

\crefname{section}{Section}{Sections}
\crefname{definition}{Definition}{Definitions}
\crefname{theorem}{Theorem}{Theorems}
\crefname{lemma}{Lemma}{Lemmas}
\crefname{corollary}{Corollary}{Corollaries}
\crefname{proposition}{Proposition}{Propositions}
\crefname{remark}{Remark}{Remarks}
\crefname{example}{Example}{Examples}
\crefname{conjecture}{Conjecture}{Conjectures}
\crefname{problem}{Problem}{Problems}
\crefname{question}{Question}{Questions}
\crefname{formula}{Formula}{Formulas}

\newcommand{\ZZ}{\mathbb{Z}}

\newcommand{\KK}{\mathbb{K}}
\newcommand{\FF}{\mathbb{F}}

\newcommand{\AAA}{\mathcal{A}}
\newcommand{\BBB}{\mathcal{B}}

\newcommand{\Der}{\operatorname{Der}}

\newcommand{\rank}{\operatorname{rank}}

\newcommand{\Sym}{\operatorname{Sym}}
\newcommand{\Hom}{\operatorname{Hom}}

\newcommand{\character}{\operatorname{char}}

\usepackage{tcolorbox}
\newcommand{\red}[1]{\textcolor{red}{#1}}

\title[On an Exponent-Preserving Duality for Coxeter Multiarrangements of Type $B_2$]{On an Exponent-Preserving Duality for Coxeter Multiarrangements of Type $B_2$}
\author[S.~Maehara]{Shota Maehara}
\email{maehara.shota.027@s.kyushu-u.ac.jp}
\address{Joint Graduate School of Mathematics for Innovation, Kyushu University, 744 Motooka Nishi-ku Fukuoka 819-0395, Japan}
\keywords{hyperplane arrangements; multiarrangements; logarithmic derivation modules}
\subjclass{32S22, 52C35, 13N15}

\begin{document}

\begin{abstract}
    In the theory of hyperplane arrangements, 
    the study of logarithmic derivation modules 
    is one of the most important subjects. 
    If the module is free 
    over a graded ring, 
    the multiset of degrees of its homogeneous basis 
    is called exponents. 
    We utilize some basic theory of linear algebra 
    to study the exponents of multiarrangements. 
    As an application, we show a new exponent-preserving duality 
    for the Coxeter multiarrangements of type $B_2$, 
    which was conjectured by Torsten Hoge. 
\end{abstract}

\maketitle


\section{Introduction}\label{section:introduction} 
Let $\KK$ be a field, 
$V:=\KK^\ell$, $V^*:=\Hom_\KK(V,\KK)$ and $S:=\Sym(V^*)$. 
Throughout the article, let us use the symbol 
$p:=\character\KK$, which is zero or a prime number. 
By fixing a $\KK$-linear basis $\{x_1,\ldots,x_\ell\}$ of $V^*$, 
we identify $S$ with a polynomial ring $\KK[x_1,\ldots,x_\ell]$. 
For the polynomial ring $S$, 
a $\KK$-linear map $\theta:S\to S$ is called a \emph{$\KK$-derivation} 
if it satisfies 
$\theta(fg)=\theta(f)g+f\theta(g)$ for all $f,g\in S$. 
According to this rule, 
it follows immediately that
$\theta(1)=\theta(1\cdot1)=\theta(1)+\theta(1)$, 
which implies $\theta(k)=0$ for any $k\in\KK$. 
Moreover, $\theta$ is determined by $\{\theta(x_j)\mid 1\leq j \leq \ell\}$. 
Let $\Der S$ denote the set of $\KK$-derivations of $S$. 
Then, $\Der S$ is a free $S$-module 
generated by 
$\left\{\partial_{x_i}:=\frac{\partial}{\partial{x_i}} \mid 1 \leq i \leq \ell\right\}$, 
where $\partial_{x_i}(x_j)$ is the Kronecker delta $\delta_{ij}$. 
We sometimes identify 
$\Der S=\bigoplus_{i=1}^\ell S\partial_{x_i}$ with $S^{\oplus\ell}$ 
under the isomorphism 
$\theta\mapsto\oplus_{i=1}^\ell \theta(x_i)$. 
In case $\ell=2,3$, we often use the coordinates 
$\{x,y,z\}$ instead of $\{x_1,x_2,x_3\}$, just for readability. 

Our main interest is the \emph{freeness} of 
a specific submodule of $\Der S$ 
called \emph{logarithmic derivation module} (see \Cref{definition:derivation module}), 
associated with \emph{hyperplane arrangements} or \emph{multiarrangements}. 
%
We call a finite set of \emph{hyperplanes}, 
codimension one subspaces of the ambient space $V=\KK^\ell$, 
an \emph{$\ell$-dimensional hyperplane arrangement} or \emph{$\ell$-arrangement} for short. 
As an easy generalization of hyperplane arrangements, 
a finite \emph{multiset} of hyperplanes is called \emph{multiarrangement}. 
We usually state a multiarrangement by a pair of 
its underlying arrangement $\AAA$ and a \emph{multiplicity} $m:\AAA\to\ZZ_{>0}$. 
In particular, when we assume the constant multiplicity 
$\mathbbm{1}:\AAA\to \{1\}$, we identify $(\AAA,\mathbbm{1})$ with 
the underlying arrangement $\AAA$. 
In this sense, an arrangement $\AAA=(\AAA,\mathbbm{1})$ is said to be \emph{simple}. 
For a multiarrangement $(\AAA,m)$, 
let $|m|:=\sum_{H\in\AAA}m(H)$ be its \emph{size}, 
or the cardinality as a multiset. 
For each hyperplane $H$, 
let $\alpha_H\in S_1$ denote the defining linear form, 
which is unique up to scalar multiple. 
Then, it is also useful to let 
$Q(\AAA,m):=\prod_{H\in\AAA}\alpha_H^{m(H)}$ 
denote the \emph{defining polynomial} 
of the multiarrangement $(\AAA,m)$. 

In the theory of hyperplane arrangements, 
the famous \emph{Terao Conjecture} 
is a central question, 
which asks whether the freeness of logarithmic derivation modules of simple arrangements, 
$D(\AAA,\mathbbm{1})$, 
is combinatorially determined (for example, see \cite{Orlik-Terao}). 
When $\ell=2$, 
the modules $D(\AAA,m)$ are known to be free for any $\AAA$ and $m$ (see \cite{Ziegler}, Corollary 7). 
On the other hand, when $\ell\geq3$, 
the modules $D(\AAA,m)$ are not free in general, including the simple case. 
For example, 
Abe \cite{Abe-deleted-A3}, 
DiPasquale--Francisco--Mermin--Schweig \cite{Dipasquale-et-al-A3}, and 
DiPasquale--Wakefield \cite{DiPasquale-Wakefield-X3} 
completely classified the multiplicities $m$ 
for which $D(\AAA,m)$ is free when the underlying $3$-arrangement $\AAA$ is 
the deleted-$A_3$, $A_3$, and $X_3$ arrangement, respectively.
%
%
As a very important criterion for Terao Conjecture, 
Yoshinaga \cite{Yoshinaga-criterion} showed that 
the freeness of any simple $3$-arrangement 
is determined by its combinatorics and 
the \emph{exponents} (the degrees of a homogeneous basis) 
of a restricted $2$-multiarrangement. 
In this article, we develop a linear-algebraic approach 
to study logarithmic derivation modules. 
As an application, we establish a new exponent-preserving duality 
for \emph{Coxeter multiarrangements of type $B_2$}, 
whose underlying arrangement is the reflection arrangement 
of the Weyl group of type $B_2$. 

\begin{definition}\label{definition:dual}
Assume $p\neq2$ 
and let $\AAA(B_2)$ be the \emph{Coxeter arrangement of type $B_2$}. 
%
%
Let $m$ be a multiplicity on $\AAA(B_2)$, 
satisfying $|m|\in2\ZZ$ and $m(H)<|m|/2$ for all $H\in\AAA(B_2)$. 
Then, we define the \emph{dual multiplicity} 
$m^\perp:\AAA(B_2)\to\ZZ_{>0}$ 
to satisfy $m^\perp(H)+m(H)=|m|/2$ for all $H\in\AAA(B_2)$. 
\end{definition}

Note that the defining polynomial of $\AAA(B_2)$ 
can be written as $Q(\AAA(B_2))=xy(x-y)(x+y)$. 
If we allow $p=2$, 
it holds that the degeneration $xy(x-y)(x+y)=xy(x+y)^2$ happens; 
see \cite{Tsujie-Uchiumi} for this case. 
Also, note that the dual operator 
$(-)^\perp$ 
preserves the size $|m|=|m^\perp|$ 
and is 
involutive as $(m^\perp)^\perp=m$. 

Our main theorem is the following, 
which was conjectured by Torsten Hoge \cite{Hoge-private} 
for $p=0$. 
\begin{theorem}[Conjectured by Hoge \cite{Hoge-private}]
\label{theorem:main}
Under the notation of \Cref{definition:dual}, if 
\begin{itemize}
    \item $p=0$\ \ or 
    \item $p\geq3$ and $|m|\leq2p+2$,  
\end{itemize}
then the exponents are preserved under the dual operation. 
\end{theorem}

\begin{example}[\cite{Hoge-private,Wiesner-private}]
\label{example:132-66-37-33}
    Assume $p=0$ and let $(\AAA(B_2),m)$ be the multiarrangement defined by 
    $Q(\AAA(B_2),m)=x^{132}y^{66}(x-y)^{37}(x+y)^{33}$. 
    Note that $|m|=132+66+37+33=268$ and 
    $Q(\AAA(B_2),m^\perp)=x^{2}y^{68}(x-y)^{97}(x+y)^{101}$. 
    By computer calculation, we have 
    $\exp(\AAA(B_2),m)=\exp(\AAA(B_2),m^\perp)=(133,135)$. 
    To check the calculation, for example, 
    see and use \cite{Macaulay2,Numata-algorithm,sagemath}. 
\end{example}
%

To prove \Cref{theorem:main}, 
we use the following proposition. 
Here we may assume any characteristic $p$ 
and any arrangement $\AAA$, not only $\AAA(B_2)$. 
See \Cref{definition:sigma} for the definition of $\sigma$. 
\begin{proposition}\label{lemma:main}
    Let $(\AAA,m)$ be an $\ell$-multiarrangement 
    and $n$ be a nonnegative integer that satisfies $n\geq m(H)$ for any $H\in\AAA$. 
    Then, we have 
    $$
    \dim_{\KK}D(\AAA,m)_n
    =\dim_{\KK}\ker\sigma+\ell\binom{n+\ell-1}{\ell-1}
    -\sum_{H\in\AAA}\left\{\binom{n+\ell-1}{\ell-1}-\binom{n-m(H)+\ell-1}{\ell-1}\right\}. 
    $$
    Especially, substituting $\ell=2$, we have 
    $$\dim_{\KK}D(\AAA,m)_n=\dim_{\KK}\ker\sigma + (2(n+1)-|m|).$$ 
\end{proposition}


For $p\geq3$, 
the bound in 
\Cref{theorem:main} is sharp 
in the following sense. 
\begin{example}
    Assume $p\geq3$ and 
    let $(\AAA(B_2),m)$ be the multiarrangement defined by 
    $Q(\AAA(B_2),m)=x^{1}y^{3}(x-y)^{p}(x+y)^{p}$. 
    Note that $|m|=2p+4 > 2p+2$ and 
    $Q(\AAA(B_2),m^\perp)=x^{p+1}y^{p-1}(x-y)^{2}(x+y)^{2}$. 
    By definition, we have 
    $\theta_F:=x^p\partial_x+y^p\partial_y\in D(\AAA(B_2),m)$ and 
    $x\theta_F\in D(\AAA(B_2),m^\perp)$. 
    An easy calculation shows that these are minimal degree logarithmic derivations for each multiplicity, 
    so 
    \Cref{theorem:Saito's criterion} implies that 
    $\exp(\AAA(B_2),m)=(p,p+4)$ but $\exp(\AAA(B_2),m^\perp)=(p+1,p+3)$. 
    Also see \Cref{proposition:border}. 
\end{example}

\begin{remark}\label{remark:difficulty-B2-exponents}
    To completely classify the exponents of multiarrangements 
    for a fixed underlying arrangement $\AAA$ 
    is very difficult, even if $\AAA=\AAA(B_2)$. 
    Although there has been some progress for the type $B_2$, such as 
    \cite{Abe:chamber,Maehara-WYmatrix,2312.06356,Wakamiko:universal} 
    with $p=0$, 
    the complete classification is still unknown. 
    Exponents are calculated 
    if the multiarrangements are \emph{extendable} 
    (obtained as a restriction of a free arrangement; 
    see \cite{Yoshinaga-extendable}) 
    with a known free $3$-arrangement, as in \cite{Abe:B_2}, 
    but it is shown in \cite{Hoge-M-Wiesner} that 
    there exist infinitely many non-extendable Coxeter multiarrangements of type $B_2$ with $p=0$. 
    When we assume $p\geq3$, the study also seems difficult. 
    An important theory is provided 
    in \cite{Tsujie-Uchiumi} for positive $p$, 
    but it especially focuses on the case $|\AAA|=3$. 
\end{remark}
    Despite the difficulty mentioned in \Cref{remark:difficulty-B2-exponents}, 
    we are able to make the $B_2$-exponents classification process 
    much shorter by the following result: 
    Owing to \Cref{theorem:main}, 
    it suffices to consider 
    either $\exp(\AAA,m)$ or $\exp(\AAA,m^\perp)$. 

\begin{corollary}\label{corollary:MinPlusMaxIsUpToHalf}
    To classify the exponents of all Coxeter multiarrangements of type $B_2$ 
    with $p=0$, 
    it suffices to consider 
    multiplicities 
    $m:\AAA(B_2)\to \ZZ_{>0}$ 
    that satisfy 
    $$\min\{m(H)\mid H\in\AAA(B_2)\}+\max\{m(H)\mid H\in\AAA(B_2)\}
    \leq |m|/2.$$ 
\end{corollary}

The rest of this article is composed as follows. 
In \Cref{section:preliminary}, 
we recall some standard notation and results 
in the theory of hyperplane arrangements. 
In \Cref{section:Proof of the main results}, we show the main result 
after building a basic but important $\KK$-linear map $\sigma$. 
In \Cref{section:Future Work}, 
we collect some open questions related to our results. 
%

\section{Preliminaries}\label{section:preliminary}
In this section, we prepare some basic results 
in the theory of hyperplane arrangements. 
First, we recall the definition of the logarithmic derivation module 
as follows. 
\begin{definition}\label{definition:derivation module}
    For an $\ell$-multiarrangement $(\AAA,m)$, we define the 
    \emph{logarithmic derivation module $D(\AAA,m)$}, 
    a submodule of $\Der S$, by 
    $\{\theta \in \Der S \mid \theta(\alpha_{H})\in\alpha_{H}^{m(H)}S 
    \text{ for all } H\in\AAA \}$. 
    Note that $D(\AAA,m)$ is a graded $S$-module 
    $D(\AAA,m)=\bigoplus_{n\geq0} D(\AAA,m)_n$. 
    When $D(\AAA,m)$ is free as an $S$-module and has a homogeneous basis 
    $\{\theta_i \mid 1 \leq i \leq \ell\}$, 
    we call $(\AAA,m)$ itself \emph{free} 
    with \emph{exponents} 
    $\exp(\AAA,m):=(\deg(\theta_1),\deg(\theta_2),\ldots,\deg(\theta_{\ell}) )$. 
\end{definition}

\begin{remark}\label{remark:exponents-leq-notation}
Recall that exponents is a multiset of degrees for a basis, 
not a tuple (see \cite{Orlik-Terao}, including its well-definedness). 
We sometimes write exponents as $(d_1,d_2,\ldots,d_\ell)_{\leq}$, 
which includes the inequality 
$d_1\leq d_2\leq\cdots \leq d_\ell$. 
In this manner, it is obvious that $d_1\leq|m|/\ell$. 
\end{remark}

The following result, usually called 
\emph{Saito's criterion} or 
\emph{Saito--Ziegler's criterion}, 
is useful to check the freeness of multiarrangements 
as long as we can obtain a candidate of a basis for the logarithmic derivation module. 
\begin{theorem}[\cite{Saito,Ziegler}] 
\label{theorem:Saito's criterion} 
    For an $\ell$-multiarrangement $(\AAA,m)$ 
    and a set of homogeneous derivations 
    $\{\theta_{i} \mid 1 \leq i \leq \ell\} \subset D(\AAA,m)$, 
    this set 
    forms a basis of $D(\AAA,m)$ if and only if 
    $\det((\theta_i(x_j))_{1 \leq i,j \leq \ell})=c\cdot Q(\AAA,m)$ 
    with some $c\in\KK\setminus\{0\}$.
    In particular, it follows that 
    $\sum_{1 \leq i \leq \ell}\deg(\theta_i)=|m|
    \left(:=\sum_{H\in\AAA}m(H)\right)$ 
    for a free multiarrangement equipped with a homogeneous basis 
    $\{\theta_{i} \mid 1 \leq i \leq \ell\}$.
\end{theorem}
Recall that logarithmic derivation modules are free 
if $\ell=2$ as mentioned in \Cref{section:introduction}. 
If a $2$-multiarrangement 
does not belong to \emph{balanced} class, 
it is easy to calculate the exponents. 
\begin{definition}\label{definition:balanced}
    When a multiplicity $m$ on a $2$-arrangement $\AAA$ satisfies 
    $m(H)<|m|/2$ for all $H\in\AAA$, 
    we call the multiplicity (and also the multiarrangement itself) \emph{balanced}. 
\end{definition}
\begin{theorem}[Case 2.1, \cite{Wakefield-Yuzvinsky}]\label{theorem:unbalanced}
    If a $2$-dimensional multiarrangement $(\AAA,m)$ is not balanced 
    with $m(H)\geq|m|/2$ for an $H\in\AAA$, 
    then we have $\exp(\AAA,m)=(m(H),|m|-m(H))$. 
\end{theorem}
Note that \Cref{theorem:unbalanced} is valid over any field $\KK$ of any characteristic $p$. 
On the other hand, for balanced $2$-multiarrangements with $p=0$, 
the next theorem is very important. 
\begin{theorem}[Theorem 1.6, \cite{Abe:chamber}]\label{theorem:upper bound of delta}
    Assume $p=0$. 
    For a balanced $2$-multiarrangement $(\AAA,m)$ 
    with $\exp(\AAA,m)=(d_1,d_2)$,  
    we have $|d_1-d_2|\leq|\AAA|-2$. 
\end{theorem}


In this article, 
we focus on the $\KK$-vector space $D(\AAA,m)_n$, 
which is the set of homogeneous degree $n$-th logarithmic derivations. 
Here, recall that 
$D(\AAA,m)_n=D(\AAA,m)\cap (\Der S)_n$. 
The following lemma is easy, but plays an important role in our proof, 
especially when we assume a positive characteristic $p$. 
In a sense, we utilize this lemma instead of \Cref{theorem:upper bound of delta}. 
\begin{lemma}\label{lemma: exponents is settled by half minus one}
    Let $(\AAA,m)$ be a $2$-multiarrangement 
    with $|m|\in2\ZZ$ and $\exp(\AAA,m)=(d_1,d_2)_{\leq}$. 
    Setting $s:=|m|/2$, it follows that 
    $\dim_\KK D(\AAA,m)_{s-1}=s-d_1$. 
    This means that the exponents are given by 
    $\exp(\AAA,m)=(s-\dim_\KK D(\AAA,m)_{s-1},s+\dim_\KK D(\AAA,m)_{s-1})$. 
\end{lemma}
\begin{proof}
    Let $\theta\in D(\AAA,m)$ be a derivation of degree $d_1$ 
    and let $k:=s-d_1$ be a nonnegative integer. 
    Since $\theta$ is a minimal degree element in $D(\AAA,m)$, 
    $\dim_\KK D(\AAA,m)_{s-1}=0$ for $k=0$, as desired. 
    For $k>0$, 
    it follows that 
    $\{x^{(k-1)-i}y^{i}\cdot\theta \mid 0\leq i \leq k-1\}\subset D(\AAA,m)_{s-1}$ 
    forms a $\KK$-basis. 
    Hence we have 
    $\dim_\KK D(\AAA,m)_{s-1}=k=s-d_1$. 
    Finally, 
    $d_2$ is calculated to satisfy $d_1+d_2=2s$ by \Cref{theorem:Saito's criterion}. 
\end{proof}

At the end of this section, 
we give a useful characterization of the logarithmic derivation module. 
\begin{definition}
    Let $(\AAA,m)$ be an $\ell$-multiarrangement. 
    Then, define an $S$-homomorphism 
    \begin{align*}    
    \Psi: \Der S \cong S^{\oplus\ell} 
    &\to \bigoplus_{H\in\AAA} S/\alpha_H^{m(H)}S.\\
    \theta \ \ 
    &\mapsto \bigoplus_{H\in\AAA} \overline{\theta(\alpha_H)}
    \end{align*}
    Similarly, for arbitrary $n\geq\max \{m(H)\mid H\in\AAA\}$, 
    define a $\KK$-linear map 
    \begin{align*}    
    \Psi_n: (\Der S)_n \cong S_n^{\oplus\ell} 
    &\to \bigoplus_{H\in\AAA} S_n/\alpha_H^{m(H)}S_{n-m(H)}.\\
    \theta \ \ 
    &\mapsto \bigoplus_{H\in\AAA} \overline{\theta(\alpha_H)}
    \end{align*}
\end{definition}
    By these morphisms, 
    the logarithmic derivation modules can be characterized as follows. 

\begin{proposition}[cf. \cite{DiPasquale:homological}, Lemma 2.2]
\label{proposition:kernel-characterization-of-log-derivation-modules}
    It follows that 
    $\ker \Psi \cong D(\AAA,m)$ as an $S$-module, and 
    $\ker \Psi_n \cong D(\AAA,m)_n$ as a $\KK$-vector space. 
\end{proposition}

\section{Proof of the main result}\label{section:Proof of the main results}

\subsection{The Border Case}
    When we show \Cref{theorem:main}, 
    the proof for the case $p=0$ 
    is applied to the case 
    $p\geq3$ and $|m|\leq2p$ 
    in exactly the same way. 
    However, when $|m|=2p+2$ with $p\geq3$, 
    the similar proof cannot be applied. 
    Hence, in this subsection, 
    we give a proof in another way for this \emph{border} case. 
    In fact, the exponents are always $(p,|m|-p)_{\leq}$. 

\begin{proposition}\label{proposition:border}
Assume $p>0$. 
Let $(\AAA,m)$ be a $2$-multiarrangement 
defined over $\FF_p$ 
that satisfies 
$\max\{m(H)\mid H\in\AAA\}\leq p \leq |m|/2$. 
Then, we have $\exp(\AAA,m)=(p,|m|-p)$. 
\end{proposition}
\begin{proof}
    Let $\theta_F:=x^p\partial_x+y^p\partial_y$. 
    Since $\theta_F(\alpha_H)=\alpha_H^p$ for any $H\in\AAA$, 
    we have $\theta_F\in D(\AAA,m)_p$. 
    Since there is no $\varphi\in\Der S$ such that 
    $\theta_F=f\cdot\varphi$ for any $f\in S\setminus\KK$, 
    \Cref{theorem:Saito's criterion} implies that 
    $\theta_F$ is a minimal degree derivation in $D(\AAA,m)$ and $\exp(\AAA,m)=(p,|m|-p)$. 
\end{proof}

As an easy corollary, we obtain a partial answer of \Cref{theorem:main}. 
\begin{corollary}\label{corollary:border}
    \Cref{theorem:main} is true for 
    $p\geq3$ and $|m|=2p,\ 2p+2$. 
\end{corollary}
\begin{proof}
    The arrangement $\AAA(B_2)$ is defined over $\FF_p$
    for each $p\geq3$ and 
    the assumption of balancedness implies $m(H)\leq p$ for all $H\in\AAA$. 
    Hence, \Cref{proposition:border} derives the exponents. 
\end{proof}

\subsection{The Remaining Cases}
For the proof, we will construct a $\KK$-linear composition 
\begin{align}
\bigoplus_{H\in\AAA}\left(\alpha_H^{m(H)}S_{n-m(H)}\right)^\perp
&\stackrel{\cong}{\longrightarrow} 
\bigoplus_{H\in\AAA}\left(S_n/(\alpha_H^{m(H)}S_{n-m(H)})\right)^*\\
&\stackrel{\cong}{\longrightarrow} 
\left(\bigoplus_{H\in\AAA}S_n/(\alpha_H^{m(H)}S_{n-m(H)})\right)^*\\
&\stackrel{\Psi_n^*}{\longrightarrow} 
(S_n^{\oplus\ell})^*\\
&\stackrel{\cong}{\longrightarrow} 
S_n^{\oplus\ell}, 
\end{align}
where the maps $(1)$, $(2)$, $(3)$ and $(4)$ are defined 
under a bilinear map on $S_n$ 
as follows. 

\begin{definition}\label{definition:linear-maps}
Let $(\AAA,m)$ be an $\ell$-multiarrangement and 
$n$ be a nonnegative integer that satisfies $n\geq m(H)$ for all $H\in\AAA$. 
Under a non-degenerate bilinear map $B:S_n\times S_n\to\KK$, 
we define the maps, 
$(1)\ \bigoplus_{H\in\AAA} \Phi$, $(2)\ u$, $(3)\ \Psi_n^*$ and $(4)\ v$, 
as follows. 
\begin{enumerate}
        \item $\bigoplus_{H\in\AAA} \Phi: 
        \oplus f_H\mapsto \oplus\Phi_{f_H}$,         
        which is isomorphic by \Cref{lemma:PerpendicularComplement}. 
        \item $u: 
        \oplus\phi_H\to\sum\phi_H$ 
        under the basic isomorphism 
        $\oplus\Hom(X,\KK)\cong\Hom(\oplus X,\KK)$ for a $\KK$-vector space $X$. 
        \item $\Psi_n^*$: the dual map of $\Psi_n$. 
        %
        \item $v: 
        \omega\mapsto\varphi$ such that 
        $\omega(\theta)=B^\ell(\varphi,\theta),\ 
        (\forall \theta\in S_n^{\oplus\ell})$, 
        which is isomorphic by \Cref{lemma:billinear}.
\end{enumerate}
\end{definition}

We move on to define $\Phi,B^\ell$ in \Cref{definition:linear-maps}. 
Throughout this section, let $X$ be a finite dimensional $\KK$-vector space. 

\begin{lemma}\label{lemma:billinear}
    Let $B: X\times X \to \KK$ be a non-degenerate bilinear map. 
    Then, for any natural number $\ell\in\ZZ_{>0}$, 
    the following bilinear map becomes non-degenerate again: 
    $$B^\ell: X^{\oplus\ell}\times X^{\oplus\ell}\to \KK;\ 
    ((x_1,\ldots,x_\ell), (y_1,\ldots,y_\ell))\mapsto 
    \sum_{i=1}^\ell B(x_i,y_i).$$
\end{lemma}
\begin{proof}
    Fix $y=(y_1,\ldots,y_\ell)\in X^{\oplus\ell}$ with $y_k\neq0$. 
    Since $B$ is non-degenerate, 
    there exists $x_k\in X$ such that 
    $B(x_k,y_k)\neq0$. Hence, setting $x:=(0,\ldots,x_k,\ldots,0)$, 
    we have $B^\ell(x,y)\neq0$. 
\end{proof}

Let $X^*:=\Hom_\KK(X,\KK)$ denote the dual space of $X$. 
For a non-degenerate bilinear map $B: X\times X\to \KK$, 
define the induced isomorphism $\Phi: X\to X^*;\ x\mapsto\Phi_x$ as 
$$\forall x,x'\in X,\  \Phi_x(x'):=B(x,x').$$
Under this notation, we have the following. 
\begin{lemma}\label{lemma:PerpendicularComplement}
For any vector subspace $Y\subset X$, we have 
$$(X/Y)^*\cong Y^\perp:=\{x\in X\mid B(x,y)=0\ (\forall y\in Y)\}.$$
\end{lemma}

\begin{proof}
By considering a natural projection 
$\pi:X\to X/Y$, 
it is obvious that 
\[
(X/Y)^*:=\Hom_\KK(X/Y,\KK)
\cong
\{\varphi\in X^*\mid \varphi(Y)=0\}. 
\]
When we identify $X\cong X^*$ by $\Phi$, 
the equation $\Phi_x(Y)=0$ is equivalent to 
$B(x,y)=0$ for any $y\in Y$, or $x\in Y^\perp$. 
Hence, we have $(X/Y)^*\cong Y^\perp$. 
\end{proof}

The following map is a key tool for \Cref{theorem:main}. 
\begin{definition}\label{definition:sigma}
    For an $\ell$-multiarrangement $(\AAA,m)$ 
    and a non-degenerate $\KK$-bilinear map $B:S_n\times S_n\to\KK$, 
    define a $\KK$-linear map 
    $$\sigma: 
    \bigoplus_{H\in\AAA}
    \left(\alpha_H^{m(H)}S_{n-m(H)}\right)^\perp \to S_n^{\oplus\ell};\ \ 
    \bigoplus_{H\in\AAA} f_H \mapsto
    \bigoplus_{i=1}^\ell
    \left(
    \sum_{H\in\AAA} f_H\cdot\partial_{x_i}(\alpha_{H})
    \right).
    $$ 
\end{definition}

\begin{lemma}\label{lemma:composition}
    The map $\sigma$ 
    is obtained as the composition of the maps in \Cref{definition:linear-maps}, 
    which implies $\rank\sigma=\rank \Psi_n^*=\rank \Psi_n$. 
    As a result, we have \Cref{lemma:main}. 
\end{lemma}
\begin{proof}
For 
$\bigoplus_{H\in\AAA} f_H\in
\bigoplus_{H\in\AAA}\left(\alpha_H^{m(H)}S_{n-m(H)}\right)^\perp$, it follows that 
\begin{align*}
    (v\circ\Psi_n^*\circ u\circ \bigoplus_{H\in\AAA} \Phi)\bigoplus_{H\in\AAA} f_H
    &=(v\circ\Psi_n^*\circ u)(\bigoplus_{H\in\AAA}\Phi_{f_H})\\
    &=(v\circ\Psi_n^*)(\sum_{H\in\AAA}\Phi_{f_H}).
\end{align*}
Here, $\Psi_n^*(\sum_{H\in\AAA}\Phi_{f_H})\in (S_n^{\oplus\ell})^*$ satisfies, 
for any $\theta=\bigoplus_{i=1}^\ell h_i\in S_n^{\oplus\ell}$, 
\begin{align*}
\Psi_n^*(\sum_{H\in\AAA}\Phi_{f_H})(\theta)
&=(\sum_{H\in\AAA}\Phi_{f_H})(\Psi_n(\theta))\\
&=(\sum_{H\in\AAA}\Phi_{f_H})(\bigoplus_{H\in\AAA}\overline{\theta(\alpha_H)})\\
&=\sum_{H\in\AAA}B(f_H,\theta(\alpha_H))\\
&=\sum_{H\in\AAA}B(f_H,\sum_{i=1}^\ell \partial_{x_i}(\alpha_H)\cdot h_i)\\
&=\sum_{H\in\AAA}\sum_{i=1}^\ell B(f_H,\partial_{x_i}(\alpha_H)\cdot h_i)\\
&=\sum_{i=1}^\ell\sum_{H\in\AAA} B(f_H \cdot\partial_{x_i}(\alpha_H),h_i)\\
&=\sum_{i=1}^\ell B(\sum_{H\in\AAA}f_H \cdot\partial_{x_i}(\alpha_H),h_i). 
\end{align*}

Since $v(\Psi_n^*(\sum_{H\in\AAA}\Phi_{f_H}))$ is the unique 
$\varphi=\bigoplus_{i=1}^\ell e_i\in S_n^{\oplus\ell}$ 
that satisfies 
\begin{align*}
\Psi_n^*(\sum_{H\in\AAA}\Phi_{f_H})(\theta)
=B^\ell(\varphi,\theta)
\iff 
\sum_{i=1}^\ell B(\sum_{H\in\AAA}f_H \cdot\partial_{x_i}(\alpha_H),h_i)
=\sum_{i=1}^\ell B(e_i,h_i) 
\end{align*}
for any $\theta\in S_n^{\oplus\ell}$, we have 
\begin{align*}
    (v\circ\Psi_n^*)(\sum_{H\in\AAA}\Phi_{f_H})
    &=\bigoplus_{i=1}^\ell
    \left(
    \sum_{H\in\AAA} f_H\cdot\partial_{x_i}(\alpha_{H})
    \right)\\
    &=\sigma(\bigoplus_{H\in\AAA} f_H) 
\end{align*}
as desired. 
%
%
    Finally, the equation of \Cref{lemma:main} holds
    by using the fundamental theorem on linear maps as 
    \begin{align*}
    \dim D(\AAA,m)_n
    &=\dim\ker \Psi_n\\
    &=\ell\cdot \binom{n+\ell-1}{\ell-1}-\rank \Psi_n\\
    &=\dim\ker\sigma+\ell\binom{n+\ell-1}{\ell-1}
    -\sum_{H\in\AAA}\left\{\binom{n+\ell-1}{\ell-1}-\binom{n-m(H)+\ell-1}{\ell-1}\right\}. 
    \end{align*}
    Here we use the fact $\dim_\KK S_n=\binom{n+\ell-1}{\ell-1}$. 
\end{proof}


When $\ell=2$, we also have the following. 
\begin{lemma}[cf. \cite{MR4752129:apolarity}-Lemma2.4, \cite{MR1735271:apolaritybook}-Lemma 1.31]
\label{lemma:key}
    Assume $\ell=2$. 
    It follows that the bilinear map 
    $B_n:S_n\times S_n\to\KK$ defined by 
    $$B_n(f,g)=g(\partial_y,-\partial_x)(f),\ \forall f,g\in S_n$$
    is non-degenerate 
    if $[p=0]$ or $[p\geq2$ and $n < p]$. 
    Under this bilinear map, for any $\alpha\in S_1\setminus\{0\}$ and any integers $m,n$ 
    satisfying $1\leq m \leq n$, we have 
    $$\left(\alpha^m S_{n-m}\right)^\perp
    =\alpha^{n-m+1} S_{m-1}.$$
\end{lemma}
\begin{proof}
    First, $B_n$ is non-degenerate since 
    \begin{equation*}
    B_n(x^iy^{n-i}, x^{n-j}y^j)
    =(-1)^{j}\cdot\partial_y^{n-j}\partial_x^j(x^iy^{n-i})
    =
    \begin{cases}
    \ \ \ \ \ \ \ \ \ \ 0 & \text{if } (i\neq j) \\
    (-1)^j\cdot j!(n-j)! & \text{if } (i=j)
    \end{cases}
    \end{equation*}
    for any monomial of degree $n$. 
    
    Let $\alpha=ax+by$ with $a,b\in\KK$ and define a derivation 
    $D_\alpha:=a\partial_y-b\partial_x \in (\Der S)_0\setminus\{0\}$. 
    Since $B_n(f,\alpha^mh)=B_{n-m}(D_\alpha^m(f),h)$ 
    for any $f\in S_{n}$ and $h\in S_{n-m}$,
    we have 
    $$\left(\alpha^m S_{n-m}\right)^\perp
    =\{f\in S_n \mid D_\alpha^m(f)=0\}. 
    $$
    Since 
    $S_n
    =\langle \alpha^n, \alpha^{n-1}x,\ldots, \alpha x^{n-1}, x^n \rangle_{\KK}$ 
    or 
    $S_n=\langle \alpha^n, \alpha^{n-1}y,\ldots, \alpha y^{n-1}, y^n \rangle_{\KK}$, 
    we have 
    $$\{f\in S_n \mid D_\alpha^m(f)=0\}
    =\alpha^{n-m+1} S_{m-1},$$
    as desired. 
\end{proof}

\begin{remark}
The bilinear map $B_n$ in \Cref{lemma:key} 
is essentially the same as the so-called 
{apolar pairing}. 
Apolarity is considered as a fundamental tool 
to study polynomial {Waring Problem}, 
which asks how many $1$-forms $\alpha_i$ are needed 
to write a degree $d$ homogeneous polynomial 
by the form $\sum \alpha_i^d$. 
\end{remark}

Now we are ready to prove the main theorem. 

\begin{proof}[Proof of \Cref{theorem:main}]
    Owing to \Cref{corollary:border}, 
    it remains to consider the case either 
    $p=0$ or $p\geq3$ and $|m|\leq2p-2$. 
    By \Cref{lemma: exponents is settled by half minus one}, 
    it suffices to show that 
    $$\dim_{\KK}D(\AAA(B_2),m)_{s-1}=\dim_{\KK}D(\AAA(B_2),m^\perp)_{s-1}$$ 
    for $s:=|m|/2$. 
    Also, let $n:=s-1$ be the nonnegative integer. 
    Since $|m|=2s=2(n+1)$, we have 
    $\dim_{\KK}D(\AAA(B_2),m)_n=\dim_{\KK}\ker\sigma$ 
    by \Cref{lemma:main}. 

    By \Cref{lemma:key}, 
    under the bilinear map $B_n$, 
    $\sigma$ is to be a linear map between 
    $$\bigoplus_{H\in\AAA}
    \left(\alpha_H^{n-m(H)+1}S_{m(H)-1}\right) 
    \to S_n\oplus S_n.$$ 
    We show the isomorphism 
    $D(\AAA(B_2),m^\perp)_n\cong\ker\sigma$. 
    For the  multiarrangement $(\AAA(B_2),m)$, 
    let us fix a defining polynomial by 
    $Q(\AAA(B_2),m)=x^ay^b(x-y)^c(x+y)^d$. 
    %
    
    First, assume 
    $\oplus f=(f_{x},f_{y},f_{x-y},f_{x+y})\in\ker\sigma$, where 
    each $f_{\alpha_H}$ belongs to 
    $\alpha_H^{n-m(H)+1}S_{m(H)-1}$ 
    for $H\in\AAA(B_2)$. 
    Since 
    $f_{\alpha_H}$ is divided by 
    $\alpha_H^{n-m(H)+1}=\alpha_H^{s-m(H)}$ for any $H\in\AAA(B_2)$ 
    and 
    $\sigma(\oplus f)
    =(f_{x}+f_{x-y}+f_{x+y}, f_{y}-f_{x-y}+f_{x+y})
    =(0,0)$, 
    we have 
    $f_x-f_y=-2f_{x-y}\in(x-y)^{s-c}S_{c-1}$ and 
    $f_x+f_y=-2f_{x+y}\in(x+y)^{s-d}S_{d-1}$, 
    which implies 
    $f_x\partial_x+f_y\partial_y \in D(\AAA(B_2),m^\perp)_n$. 
    
    Conversely, for arbitrary 
    $f_x\partial_x+f_y\partial_y \in D(\AAA(B_2),m^\perp)_n$,
    we have $(f_x,f_y,f_{x-y},f_{x+y})\in\ker\sigma$ 
    by setting 
    $f_{x-y}:=-\frac{1}{2}(f_x-f_y)$ and 
    $f_{x+y}:=-\frac{1}{2}(f_x+f_y)$. 
    
    Since these two $\KK$-linear maps are inverse to each other, 
    we have 
    $D(\AAA(B_2),m^\perp)_{n} \cong \ker\sigma$ 
    as $\KK$-vector spaces. 
    Therefore, it follows that 
    $\dim_{\KK}D(\AAA(B_2),m)_n=\dim_{\KK}D(\AAA(B_2),m^\perp)_n
    =\dim_{\KK}\ker\sigma$, as desired. 
\end{proof}

\section{Future Work}\label{section:Future Work}

The complete classification of the exponents 
for Coxeter multiarrangements 
remains open, 
except for those of type $A_2$ by Wakamiko \cite{Wakamiko}. 
Focusing on the type $B_2$, 
a computer calculation might suggest the following phenomenon 
for $p=0$. 
\begin{conjecture}
    Let $(\AAA(B_2),m)$ be a Coxeter multiarrangement of type $B_2$. 
    If the following conditions are satisfied, 
    then $m$ is not a \emph{peak point} 
    (see Abe--Numata \cite{Abe-Numata} for definition). 
    \begin{itemize}
        \item $m$ is balanced, $|m|$ is even, 
        \item {$3\leq m(H)\leq |m|/2-3$} for all $H\in\AAA(B_2)$ and 
        \item $\min\{|m(H_x)-m(H_y)|, |m(H_{x-y})-m(H_{x+y})|\}\geq4$. 
    \end{itemize}
\end{conjecture}

In contrast, 
for $p\geq3$, there are much fewer results for $\exp(\AAA(B_2),m)$. 
\begin{problem}
    For $p\geq3$ and $|m|>2p+2$, 
    when does the 
    exponent-preserving duality hold?
\end{problem}

\begin{problem}
    For some $p\geq3$, 
    is it possible to completely classify the $B_2$-exponents 
    before we complete the classification for $p=0$?
\end{problem}

It might also be important 
to give another proof for \Cref{theorem:main}; 
there can exist more natural proofs which lead to further applications 
for other underlying arrangements. 
\begin{problem}
    Does there exist such a new $\exp$-preserving duality 
    for other arrangements?
\end{problem}


\section*{Acknowledgements}
\Cref{theorem:main} was first conjectured by Torsten Hoge 
when the author, Hoge and Sven Wiesner 
were discussing the classification of $B_2$-exponents. 
The insights of Hoge and Wiesner have improved the work a lot. 
The author thanks these two collaborators very much. 
The author also thanks Takuro Abe for his many helpful comments. 
Furthermore, a generative AI ChatGPT assisted this study, 
including a first idea to approach \Cref{theorem:main}. 
The author was supported by WISE program (MEXT) at Kyushu University, 
and was also partially supported by 
JSPS KAKENHI Grant Number JP25K24692. 
\bibliographystyle{plain}
\bibliography{999_bib_mr,998_bib_arxiv}

@Misc{Macaulay2,
          author = {Grayson, Daniel R. and Stillman, Michael E.},
          title = {Macaulay2, a software system for research in algebraic geometry},
          howpublished = {Available at \url{http://www2.macaulay2.com}}
        }

@manual{sagemath,
  Key          = {SageMath},
  Author       = {{The Sage Developers}},
  Title        = {{S}ageMath, the {S}age {M}athematics {S}oftware {S}ystem ({V}ersion 10.9)},
  note         = {available at \url{https://www.sagemath.org}},
  Year         = {2026},
}

@misc{Numata-algorithm,
Author = {Yasuhide Numata},
Title = {An Algorithm to Construct A Basis for the Module of Logarithmic Vector Fields},
Year = {2007},
Eprint = {arXiv:0707.0004},
howpublished = {arXiv:0707.0004}
}

@misc{Tsujie-Uchiumi,
Author = {Shuhei Tsujie and Ryo Uchiumi},
Title = {Exponents of 2-multiarrangements of three lines over fields of positive characteristic},
Year = {2025},
Eprint = {arXiv:2506.07070},
howpublished = {arXiv:2506.07070}
}

@misc{Maehara-WYmatrix,
Author = {Shota Maehara},
Title = {Exponents of $2$-multiarrangements and {W}akefield--{Y}uzvinsky matrices},
Year = {2025},
Eprint = {arXiv:2509.16569},
howpublished = {arXiv:2509.16569}
}

@article {Hoge-private,
    AUTHOR = {Torsten Hoge},
     TITLE = {private communications},
   JOURNAL = {\texttt{torsten.hoge@rub.de}},
     Email = {torsten.hoge@rub.de},
}

@article {Wiesner-private,
    AUTHOR = {Sven Wiesner},
     TITLE = {private communications},
   JOURNAL = {\texttt{sven.wiesner@rub.de}},
     Email = {sven.wiesner@rub.de},
}

@article {Saito,
    AUTHOR = {Saito, Kyoji},
     TITLE = {Theory of logarithmic differential forms and logarithmic
              vector fields},
   JOURNAL = {J. Fac. Sci. Univ. Tokyo Sect. IA Math.},
  FJOURNAL = {Journal of the Faculty of Science. University of Tokyo.
              Section IA. Mathematics},
    VOLUME = {27},
      YEAR = {1980},
    NUMBER = {2},
     PAGES = {265--291},
      ISSN = {0040-8980},
   MRCLASS = {32G11 (14D05 32B30)},
  MRNUMBER = {586450},
MRREVIEWER = {Zoghman\ Mebkhout},
}

@book {Orlik-Terao,
    AUTHOR = {Orlik, Peter and Terao, Hiroaki},
     TITLE = {Arrangements of hyperplanes},
    SERIES = {Grundlehren der mathematischen Wissenschaften [Fundamental
              Principles of Mathematical Sciences]},
    VOLUME = {300},
 PUBLISHER = {Springer-Verlag, Berlin},
      YEAR = {1992},
     PAGES = {xviii+325},
      ISBN = {3-540-55259-6},
   MRCLASS = {52B30 (14F35 20F36 20F55 32S25 57N65)},
  MRNUMBER = {1217488},
MRREVIEWER = {Michel\ Yves\ Jambu},
       DOI = {10.1007/978-3-662-02772-1},
       URL = {https://doi.org/10.1007/978-3-662-02772-1},
}

@incollection {Ziegler,
    AUTHOR = {Ziegler, G\"{u}nter M.},
     TITLE = {Multiarrangements of hyperplanes and their freeness},
 BOOKTITLE = {Singularities ({I}owa {C}ity, {IA}, 1986)},
    SERIES = {Contemp. Math.},
    VOLUME = {90},
     PAGES = {345--359},
 PUBLISHER = {Amer. Math. Soc., Providence, RI},
      YEAR = {1989},
      ISBN = {0-8218-5096-2},
   MRCLASS = {32C40 (05B35)},
  MRNUMBER = {1000610},
MRREVIEWER = {Hiroaki\ Terao},
       DOI = {10.1090/conm/090/1000610},
       URL = {https://doi.org/10.1090/conm/090/1000610},
}

@article {Yoshinaga-criterion,
    AUTHOR = {Yoshinaga, Masahiko},
     TITLE = {On the freeness of 3-arrangements},
   JOURNAL = {Bull. London Math. Soc.},
  FJOURNAL = {The Bulletin of the London Mathematical Society},
    VOLUME = {37},
      YEAR = {2005},
    NUMBER = {1},
     PAGES = {126--134},
      ISSN = {0024-6093,1469-2120},
   MRCLASS = {52C35 (32S22)},
  MRNUMBER = {2105827},
MRREVIEWER = {Graham\ C.\ Denham},
       DOI = {10.1112/S0024609304003704},
       URL = {https://doi.org/10.1112/S0024609304003704},
}

@article {Abe:chamber,
    AUTHOR = {Abe, Takuro},
     TITLE = {Chambers of 2-affine arrangements and freeness of
              3-arrangements},
   JOURNAL = {J. Algebraic Combin.},
  FJOURNAL = {Journal of Algebraic Combinatorics. An International Journal},
    VOLUME = {38},
      YEAR = {2013},
    NUMBER = {1},
     PAGES = {65--78},
      ISSN = {0925-9899,1572-9192},
   MRCLASS = {52C35},
  MRNUMBER = {3070120},
MRREVIEWER = {Kristopher\ Williams},
       DOI = {10.1007/s10801-012-0393-x},
       URL = {https://doi.org/10.1007/s10801-012-0393-x},
}

@article {Abe:B_2,
    AUTHOR = {Abe, Takuro},
     TITLE = {The stability of the family of {$B_2$}-type arrangements},
   JOURNAL = {Comm. Algebra},
  FJOURNAL = {Communications in Algebra},
    VOLUME = {37},
      YEAR = {2009},
    NUMBER = {4},
     PAGES = {1193--1215},
      ISSN = {0092-7872,1532-4125},
   MRCLASS = {32S22 (14J60)},
  MRNUMBER = {2510979},
MRREVIEWER = {Daniel\ C.\ Cohen},
       DOI = {10.1080/00927870802465969},
       URL = {https://doi.org/10.1080/00927870802465969},
}

@article {Wakefield-Yuzvinsky,
    AUTHOR = {Wakefield, Max and Yuzvinsky, Sergey},
     TITLE = {Derivations of an effective divisor on the complex projective
              line},
   JOURNAL = {Trans. Amer. Math. Soc.},
  FJOURNAL = {Transactions of the American Mathematical Society},
    VOLUME = {359},
      YEAR = {2007},
    NUMBER = {9},
     PAGES = {4389--4403},
      ISSN = {0002-9947,1088-6850},
   MRCLASS = {14N20 (13N15 52C35)},
  MRNUMBER = {2309190},
MRREVIEWER = {Daniel\ Matei},
       DOI = {10.1090/S0002-9947-07-04222-5},
       URL = {https://doi.org/10.1090/S0002-9947-07-04222-5},
}

@article {Abe-Numata,
    AUTHOR = {Abe, Takuro and Numata, Yasuhide},
     TITLE = {Exponents of 2-multiarrangements and multiplicity lattices},
   JOURNAL = {J. Algebraic Combin.},
  FJOURNAL = {Journal of Algebraic Combinatorics. An International Journal},
    VOLUME = {35},
      YEAR = {2012},
    NUMBER = {1},
     PAGES = {1--17},
      ISSN = {0925-9899,1572-9192},
   MRCLASS = {52C35},
  MRNUMBER = {2873095},
MRREVIEWER = {Kelly\ Jeanne\ Pearson},
       DOI = {10.1007/s10801-011-0291-7},
       URL = {https://doi.org/10.1007/s10801-011-0291-7},
}

@article {Wakamiko,
    AUTHOR = {Wakamiko, Atsushi},
     TITLE = {On the exponents of 2-multiarrangements},
   JOURNAL = {Tokyo J. Math.},
  FJOURNAL = {Tokyo Journal of Mathematics},
    VOLUME = {30},
      YEAR = {2007},
    NUMBER = {1},
     PAGES = {99--116},
      ISSN = {0387-3870},
   MRCLASS = {32S22 (52C35)},
  MRNUMBER = {2328057},
MRREVIEWER = {Daniel\ C.\ Cohen},
       DOI = {10.3836/tjm/1184963649},
       URL = {https://doi.org/10.3836/tjm/1184963649},
}

@article {Wakamiko:universal,
    AUTHOR = {Wakamiko, Atsushi},
     TITLE = {Bases for the derivation modules of two-dimensional
              multi-{C}oxeter arrangements and universal derivations},
   JOURNAL = {Hokkaido Math. J.},
  FJOURNAL = {Hokkaido Mathematical Journal},
    VOLUME = {40},
      YEAR = {2011},
    NUMBER = {3},
     PAGES = {375--392},
      ISSN = {0385-4035},
   MRCLASS = {13N15 (32C38 32S22)},
  MRNUMBER = {2883497},
MRREVIEWER = {Michael\ Gr.\ Voskoglou},
       DOI = {10.14492/hokmj/1319595862},
       URL = {https://doi.org/10.14492/hokmj/1319595862},
}

@article {MR0608532,
    AUTHOR = {Terao, Hiroaki},
     TITLE = {Generalized exponents of a free arrangement of hyperplanes and
              {S}hepherd-{T}odd-{B}rieskorn formula},
   JOURNAL = {Invent. Math.},
  FJOURNAL = {Inventiones Mathematicae},
    VOLUME = {63},
      YEAR = {1981},
    NUMBER = {1},
     PAGES = {159--179},
      ISSN = {0020-9910},
   MRCLASS = {32C40 (05B25 51F15 52A45)},
  MRNUMBER = {608532},
MRREVIEWER = {P. Orlik},
       DOI = {10.1007/BF01389197},
       URL = {https://doi.org/10.1007/BF01389197},
}

@article {2312.06356,
    AUTHOR = {Maehara, Shota and Numata, Yasuhide},
     TITLE = {Explicit description of bases for derivations of some
              {C}oxeter multiarrangements of type {$B_2$}},
   JOURNAL = {Comm. Algebra},
  FJOURNAL = {Communications in Algebra},
    VOLUME = {53},
      YEAR = {2025},
    NUMBER = {11},
     PAGES = {4718--4740},
      ISSN = {0092-7872,1532-4125},
   MRCLASS = {52C35},
  MRNUMBER = {4950644},
MRREVIEWER = {So\ Yamagata},
       DOI = {10.1080/00927872.2025.2497421},
       URL = {https://doi.org/10.1080/00927872.2025.2497421},
}

@article {Feigin-Wang-Yoshinaga,
    AUTHOR = {Feigin, Misha and Wang, Zixuan and Yoshinaga, Masahiko},
     TITLE = {Integral expressions for derivations of multiarrangements},
   JOURNAL = {Internat. J. Math.},
  FJOURNAL = {International Journal of Mathematics},
    VOLUME = {36},
      YEAR = {2025},
    NUMBER = {4},
     PAGES = {Paper No. 2450085, 20},
      ISSN = {0129-167X,1793-6519},
   MRCLASS = {52C35 (20F55)},
  MRNUMBER = {4858792},
MRREVIEWER = {Weili\ Guo},
       DOI = {10.1142/S0129167X2450085X},
       URL = {https://doi.org/10.1142/S0129167X2450085X},
}

@article{Hoge-M-Wiesner,
Author = {Torsten Hoge and Shota Maehara and Sven Wiesner},
Title = {Extendability of the {$B_2$-}arrangement},
JOURNAL = 
{Ann. Comb.},
Year = {2026}, 
Note = {\url{https://doi.org/10.1007/s00026-026-00869-z}},
}

@article {DiPasquale:homological,
    AUTHOR = {DiPasquale, Michael},
     TITLE = {A homological characterization for freeness of
              multi-arrangements},
   JOURNAL = {Math. Ann.},
  FJOURNAL = {Mathematische Annalen},
    VOLUME = {385},
      YEAR = {2023},
    NUMBER = {1-2},
     PAGES = {745--786},
      ISSN = {0025-5831,1432-1807},
   MRCLASS = {14Q10 (13D02 13N15 13P20 14N20 52C35)},
  MRNUMBER = {4542730},
MRREVIEWER = {Joseph\ Kung},
       DOI = {10.1007/s00208-021-02357-6},
       URL = {https://doi.org/10.1007/s00208-021-02357-6},
}

@incollection {Yoshinaga-extendable,
    AUTHOR = {Yoshinaga, Masahiko},
     TITLE = {On the extendability of free multiarrangements},
 BOOKTITLE = {Arrangements, local systems and singularities},
    SERIES = {Progr. Math.},
    VOLUME = {283},
     PAGES = {273--281},
 PUBLISHER = {Birkh\"{a}user Verlag, Basel},
      YEAR = {2010},
      ISBN = {978-3-0346-0208-2},
   MRCLASS = {52C35 (32S22)},
  MRNUMBER = {3025868},
MRREVIEWER = {Kelly\ Jeanne\ Pearson},
       DOI = {10.1007/978-3-0346-0209-9\{_}11}

@article {Abe-deleted-A3,
    AUTHOR = {Abe, Takuro},
     TITLE = {Free and non-free multiplicity on the deleted {$A_3$}
              arrangement},
   JOURNAL = {Proc. Japan Acad. Ser. A Math. Sci.},
  FJOURNAL = {Japan Academy. Proceedings. Series A. Mathematical Sciences},
    VOLUME = {83},
      YEAR = {2007},
    NUMBER = {7},
     PAGES = {99--103},
      ISSN = {0386-2194},
   MRCLASS = {32S22},
  MRNUMBER = {2361419},
MRREVIEWER = {Graham\ C.\ Denham},
       URL = {http://projecteuclid.org/euclid.pja/1200672008},
}

@article {Dipasquale-et-al-A3,
    AUTHOR = {DiPasquale, Michael and Francisco, Christopher A. and Mermin,
              Jeffrey and Schweig, Jay},
     TITLE = {Free and non-free multiplicities on the {$A_3$} arrangement},
   JOURNAL = {J. Algebra},
  FJOURNAL = {Journal of Algebra},
    VOLUME = {544},
      YEAR = {2020},
     PAGES = {498--532},
      ISSN = {0021-8693,1090-266X},
   MRCLASS = {32S22 (52C35)},
  MRNUMBER = {4031133},
MRREVIEWER = {Michele\ Torielli},
       DOI = {10.1016/j.jalgebra.2019.10.029},
       URL = {https://doi.org/10.1016/j.jalgebra.2019.10.029},
}

@article {DiPasquale-Wakefield-X3,
    AUTHOR = {DiPasquale, Michael and Wakefield, Max},
     TITLE = {Free multiplicities on the moduli of {$X_3$}},
   JOURNAL = {J. Pure Appl. Algebra},
  FJOURNAL = {Journal of Pure and Applied Algebra},
    VOLUME = {222},
      YEAR = {2018},
    NUMBER = {11},
     PAGES = {3345--3359},
      ISSN = {0022-4049,1873-1376},
   MRCLASS = {52C35},
  MRNUMBER = {3806729},
MRREVIEWER = {Vladimir\ N.\ Gruji\'{c}},
       DOI = {10.1016/j.jpaa.2017.12.011},
       URL = {https://doi.org/10.1016/j.jpaa.2017.12.011},
}

@article {MR4752129:apolarity,
    AUTHOR = {Blekherman, Grigoriy and Chen, Justin and Jung, Jaewoo},
     TITLE = {Sums of squares, {H}ankel index and almost real rank},
   JOURNAL = {Forum Math. Sigma},
  FJOURNAL = {Forum of Mathematics. Sigma},
    VOLUME = {12},
      YEAR = {2024},
     PAGES = {Paper No. e70, 21},
      ISSN = {2050-5094},
   MRCLASS = {14N07 (14P10 52A20)},
  MRNUMBER = {4752129},
MRREVIEWER = {Juan\ C.\ Migliore},
       DOI = {10.1017/fms.2024.45},
       URL = {https://doi.org/10.1017/fms.2024.45},
}

@article{MR1735271:apolaritybook,
author="Anthony, Iarrobino and Vassil, Kanev",
title="{P}ower {S}ums, {G}orenstein {A}lgebras, and {D}eterminantal {L}oci",
journal="Lecture Notes in Mathematics $1721$",
ISSN="0075-8434",
publisher="Springer Berlin Heidelberg",
year="1999",
DOI="10.1007/bfb0093426",
URL="https://cir.nii.ac.jp/crid/1360017289989645824"
}

\end{document}